\documentclass[12pt]{amsart}
\usepackage[pdftex]{hyperref}
\hypersetup{
	colorlinks   = true, 
	urlcolor     = gray, %Color for external hyperlinks
	linkcolor    = red, %Color of internal links
	citecolor   =  blue%Color of citations
}
\usepackage{amsmath,amsfonts,amsthm,amssymb,graphicx,xcolor}
\usepackage[all,2cell,ps]{xy}
\hypersetup{colorlinks=true, citecolor=[rgb]{0,0.5,1}}
\usepackage[margin=1.2 in, marginpar=1 in]{geometry}

\theoremstyle{plain}
\newtheorem{theorem}{Theorem}[section]
\newtheorem{corollary}[theorem]{Corollary}
\newtheorem{lemma}[theorem]{Lemma}
\newtheorem{proposition}[theorem]{Proposition}

\theoremstyle{definition}

\newtheorem{remark}[theorem]{Remark}

\newtheorem{question}[theorem]{Question}

\newcommand{\IB}{\mathbb{B}}
\newcommand{\IC}{\mathbb{C}}

\newcommand{\IF}{\mathbb{F}}

\newcommand{\IN}{\mathbb{N}}

\newcommand{\IP}{\mathbb{P}}

\newcommand{\IR}{\mathbb{R}}

\newcommand{\IZ}{\mathbb{Z}}

\DeclareMathOperator{\Hom}{Hom}

\DeclareMathOperator{\sign}{sign}
\DeclareMathOperator{\Ric}{Ric}
\DeclareMathOperator{\Ends}{Ends}
\DeclareMathOperator{\Int}{Int}

\title[Impossibility of complex-hyperbolic Einstein Dehn filling]{On the impossibility of four-dimensional complex-hyperbolic Einstein Dehn filling}
\author{Luca F. Di Cerbo}
\email{ldicerbo@ufl.edu}
\address{Department of Mathematics, University of Florida, Gainesville, United States}

\author{Marco Golla}
\email{marco.golla@univ-nantes.fr}
\address{CNRS, Laboratoire Jean Leray, Nantes University, Nantes, France}
\date{}

\begin{document}
	
	\maketitle
	
	%%%%%%%%%%%%%%%%%%%%
	\begin{abstract}
		We show that the complex-hyperbolic Einstein Dehn filling compactification cannot possibly be performed in dimension four.
	\end{abstract}
		
	%%%%%%%%%%%%%%%%%%%%
	
	\tableofcontents
		
	%%%%%%%%%%%%%%%%%%%%%%%%%%%%%%%%
	
	\section{Introduction and statement of the main result}\label{intro}

	Hyperbolic Dehn filling is a remarkable construction, due to Thurston~\cite{Thurston}, that allows to build hyperbolic metrics on closed 3-manifolds starting from non-compact complete hyperbolic 3-manifolds of finite volume. Anderson~\cite{And06} adapted this construction to higher dimension, allowing to build \emph{Einstein} metrics on $n$-manifolds starting from non-compact hyperbolic $n$-manifolds of finite volume, whose ends (or \emph{cusps}) are all diffeomorphic to $T^{n-1}\times \IR^+$, where $T^{n-1}$ is a $(n-1)$-dimensional torus. This construction is called \emph{hyperbolic Einstein Dehn filling}.
	Given the nature of the arguments in~\cite{And06}, there was hope that a similar construction could work starting from a \emph{complex-hyperbolic} $n$-manifold (note that $n$ is the complex dimension).
	%\MG{Forse dovremmo dire \emph{qui} cosa vogliamo dire con ``complex hyperbolic'', almeno a parole? Tipo i.e. ball quotients}
	This problem is for example mentioned in the survey~\cite[Page~28]{And10}. We show that such a construction is not allowed, i.e., that \emph{complex-hyperbolic Einstein Dehn filling} on ball quotients in dimension four is impossible.
	
	It is well-known that complex-hyperbolic surfaces have infra-nilmanifold cusp cross-sections. Torus-like cusps, i.e., cusps whose ideal boundary fibers over a torus, are parametrized by their \emph{Euler number}, which is a positive integer $e > 0$. Moreover, the result of Dehn filling a torus-like cusp is uniquely determined up to diffeomorphism (see Proposition~\ref{p:uniquefilling}). If a complex-hyperbolic surface has only torus-like cusps, we will call the result of filling all cusps the \emph{Dehn filling compactification} of the surface.
	
	\begin{theorem}\label{main-compact}
	For each positive integer $e$ there exists a complex-hyperbolic surface $X_e$ with cusps, all torus-like and with Euler number $e$, whose Dehn filling compactification does not admit any Einstein metric.
	\end{theorem}
	
	The nature of the proof of the theorem is topological. We will show that the closed $4$-manifolds obtained by filling the cusps of $X_e$ violates the Hitchin--Thorpe inequality, and therefore cannot support an Einstein metric. In fact, since the Hitchin--Thorpe inequality only depends on the homotopy type of the $4$-manifold, this is true for any $4$-manifold homeomorphic to the Dehn filling compactification.
	
	Based on similar ideas, we prove the following non-compact version of Theorem~\ref{main-compact}. In this case, though, we bring in more geometry: Dai and Wei's logarithmic version of the Hitchin--Thorpe inequality and Cheeger and Gromoll's splitting theorem.
	
	\begin{theorem}\label{main-noncompact}
	For each integer $e > 1$ there exists a complex-hyperbolic surface $Y_e$ with $e+3$ cusps, all torus-like, three of which have Euler number $e$ and the rest with Euler number $1$, such the $4$-manifold obtained from $Y_e$ by Dehn filling the first three cusps does not admit an Einstein metric with fibered cusp structure at infinity.
	\end{theorem}

%{\color{blue}	
%	\begin{theorem}
%		There are examples of complex-hyperbolic surfaces with ANY nilmanifolds cusps cross sections, such that any of their Dehn fillings violate the Hitchin-Thorpe inequality and as such cannot support Einstein metrics.
%	\end{theorem}
%}
	
	The remarkable property of the surfaces $X_e$ and $Y_e$, that ultimately makes them useful to prove the non-existence of Einstein Dehn filling compactifications, is that they admit smooth toroidal compactifications with \emph{non-nef} canonical divisors. Recall that a toroidal compactification is a Dehn filling compactification that minimally and uniquely compactifies the complex structure of the complex-hyperbolic surface with nilmanifold cusps, see \cite[Section~1.1]{JLMS} and \cite[Proposition 2.3]{DiCerbo2} for more details. We note that such examples cannot exist if the complex dimension is bigger than or equal to three, as shown by G. Di Cerbo and the first author~\cite{DiCerbo2}. Thus, somewhat interestingly, this paper highlights a new corollary of the fact that complex-hyperbolic geometry is special in complex dimension two. This peculiarity, when combined with the Hitchin--Thorpe inequality, is ultimately responsible for the non-existence of the complex-hyperbolic Einstein Dehn filling in real dimension four.

	We point out that most complex-hyperbolic surfaces with cusps do \emph{not} satisfy Theorem~\ref{main-compact}.
	%\MG{Reference to \ref{main-compact} instead of the old theorem}
	Indeed, as proved in \cite[Theorem~A]{DiCerbo}, \emph{most} complex-hyperbolic surfaces with cusps have smooth toroidal compactifications with \emph{ample} canonical divisors, and they then support K\"ahler--Einstein metrics thanks to the celebrated work of Yau~\cite{YauFields}. The same is true in higher dimensions as well, in fact up to a finite \'etale cover any complex hyperbolic $n$-manifold admits a smooth toroidal compactification with ample canonical divisor \cite[Theorem 1.3]{DiCerbo2}. Again by Yau~\cite{YauFields}, the smooth toroidal compactification of the cover supports a K\"ahler--Einstein metric. Remarkably, Bakker and Tsimerman~\cite{BakkerTsimerman} have recently shown that \emph{any} smooth toroidal compactification of a complex hyperbolic $n$-manifold has ample canonical divisor if $n\geq 6$. Thus, in this range of dimensions there is no need to pass to a cover to equip a smooth toroidal compactification with a K\"ahler--Einstein metric.  

	\subsection*{Organization}

	In Section~\ref{s:background} we give some further background and motivation, and establish terminology and notation.
	In Section~\ref{mainthm} we recall some background on the Hitchin--Thorpe inequality and its logarithmic analogue. In Section~\ref{part1} we give the proof of Theorem~\ref{main-compact}, based on the Hitchin--Thorpe inequality, and in Section~\ref{alternative} we give the proof of Theorem~\ref{main-noncompact}, based on the Dai--Wei inequality.

	\subsection*{Acknowledgments}
The authors would like to thank the Mathematics Department of Stony Brook University for the ideal research environment they enjoyed during the 2010/2011 and 2017/2018 academic years, as well as the Simons Center for Geometry and Physics for the outstanding coffee breaks where this collaboration began, unbeknownst to the authors. They also thank Stefano Riolo for suggesting Proposition~\ref{p:uniquefilling}, and the referee for pertinent and constructive comments that helped us
%\MG{`helped us \emph{to} improve': ho tolto il `to'}
 improve the presentation. LFDC would like to thank Michael Anderson for useful email exchanges on this topic, and Claude LeBrun for suggesting to study Anderson's Dehn filling in dimension four. MG would like to thank Vestislav Apostolov, Yann Rollin, and Fran\c{c}ois Laudenbach. LFDC is partially supported by the NSF grant DMS-2104662.

%	The authors would like to thank the Mathematics Department of Stony Brook University for the ideal research environment they enjoyed during the 2010/2011 and 2017/2018 academic years, as well as the Simons Center for Geometry and Physics for the outstanding coffee breaks where this collaboration began, unbeknownst to the authors. They also thank Stefano Riolo and Fran\c{c}ois Laudenbach for help with Proposition~\ref{p:uniquefilling}. LFDC would like to thanks Michael Anderson for useful email exchanges on this topic, and Claude LeBrun for suggesting to study Anderson's Dehn filling in dimension four. MG would like to thank Vestislav Apostolov, Yann Rollin for encouragement. LFDC is partially supported by the NSF grant DMS-2104662.

	\section{Background and motivation}\label{s:background}
	Suppose that $X$ is an $n$-manifold of finite type; a codimension-0 submanifold $E \subset X$ is a an \emph{end} (or \emph{cusp}) if:
	\begin{itemize}
		\item $E$ is a cylinder, i.e., it is diffeomorphic to $Y\times [0,\infty)$ for some closed $(n-1)$-manifold $Y$, and
		\item there is a compact submanifold $K \subset X$ such that $E$ is the closure of one of the components of $X \setminus \Int(K)$.
	\end{itemize}
	We call $Y$ the \emph{ideal boundary} of $E$, and the disjoint union of the ideal boundaries of all ends of $X$ the \emph{ideal boundary} of $X$.
	
	We can \emph{truncate} a cusp by removing $Y \setminus (1,\infty)$ from $E \cong Y \times [0,\infty)$. This produces a manifold $X_{\rm tr}$ with one less end than $X$ and with a boundary component diffeomorphic to $Y$.
	A cusp is \emph{toral} if its ideal boundary $Y$ is an $(n-1)$-torus. We say that it is torus-like if $Y$ is a circle bundle over an $(n-2)$-torus. If $E$ is torus-like, then $Y$ is the boundary of a $2$-disc bundle $D_E$ over $T^{n-2}$.
	
	By \emph{Dehn filling} along a torus-like end $E$ of $X$ we mean the following: first we truncate $E$ to obtain a boundary component of $X_{\rm tr}$ diffeomorphic to $Y$, and we glue in the $2$-disk bundle $D_E$ to obtain a manifold $\overline X$ which has one less end than $X$. (Here $\partial D_E$ is $Y$ with its orientation reversed.) Such a gluing is determined by the choice of an orientation-reversing diffeomorphism $\partial D_E \to \partial E$. %Up to diffeomorphisms of $D_E$, such a diffeomorphism is in turn determined by the image of the boundary of a $D^2$-fiber of $D_E \to T^{n-2}$, which we think of as a geodesic $\sigma \subset Y \subset \partial X_{\rm tr}$. 
	Another way of looking at Dehn filling is to view $X$ as the complement of an embedded $(n-2)$-torus $T \subset \overline X$ (i.e., the $0$-section of $D_E$).

	Broadly speaking, the question we are interested in is whether, given a geometric (e.g., hyperbolic or complex-hyperbolic) structure on $X$, one can find another geometric (e.g., hyperbolic or Einstein) structure on $\overline X$.
	
	Let us first look at hyperbolic manifolds. First, recall that the ideal boundary of a hyperbolic $n$-manifold is \emph{flat}.
	If an orientable non-compact 3-manifold $M$ admits a complete real-hyperbolic metric of finite volume, then its ideal boundary is a collection of tori (since tori are the only orientable flat surfaces). It is a well-known fact that for each cusp there are at most \emph{finitely many} slopes $s$ such that Dehn filling along $s$ produces a manifold that has no complete hyperbolic metric of finite volume. The proof of this theorem was outlined by Thurston in his celebrated Princeton University notes~\cite{Thurston}, and then more details were provided by Neumann and Zagier (in the presence of an ideal triangulation)~\cite{NeumannZagier} and by Petronio and Porti (in the general case)~\cite{PetronioPorti}.
	
	By contrast, for $n > 3$, one \emph{cannot}, as in the 3-dimensional case, construct hyperbolic metrics on all but finitely many Dehn fillings of a complete, finite-volume hyperbolic $n$-manifold with an $(n-1)$-torus end. This is a consequence of Gromov and Thurston's so-called $2\pi$ Theorem (\cite[Section~2]{And06}). To see this, suppose that $X$ is a hyperbolic $n$-manifold with a complete hyperbolic metric of finite volume, where $n > 3$. The $2\pi$ Theorem ensures that any \emph{sufficiently large} Dehn filling $\overline X$ of $X$ admits a metric of non-positive sectional curvature in which the core $T$ of the filling (which is an $(n-2)$-torus) is a totally geodesic submanifold. (Here by \emph{sufficiently large} we mean that the surgery geodesic $\sigma$ is sufficiently long.) In particular, $\pi_1(T) \cong \IZ^{n-2}$ injects into $\pi_1(\overline X)$, which therefore contains a free Abelian subgroup of rank $2$. By Preissman's theorem~\cite{Preissman} (see also~\cite[Chapter~12]{doC92}), such Dehn-filled manifolds cannot support a real-hyperbolic metric.
	
	With that said, it is remarkable that Anderson~\cite{And06} was able to extend many features of Thurston's Dehn surgery for hyperbolic $3$-manifolds to higher dimensions by softening the hyperbolicity requirement on the Dehn-filled manifold. (Biquard~\cite{Biquard1} and Bamler~\cite{Bamler} have since strengthened this construction.) More precisely, Anderson requires the Dehn-filled manifold not to be hyperbolic, but rather to have an Einstein metric with negative cosmological constant. Recall that \emph{Einstein metrics} are Riemannian metrics $g$ for which the Ricci tensor is proportional to the metric
	\[
	\Ric_g=\lambda g
	\]
	for some constant $\lambda\in\IR$, where $\lambda$ is referred to as the \emph{cosmological} or \emph{Einstein constant}. We say that an Einstein manifold is negative, positive, or flat, according to whether $\lambda$ is negative, positive, or $0$.
	In dimensions at most three, being Einstein is equivalent to having constant sectional curvature. In higher dimensions, Einstein metrics are usually hard to construct. However, they are a very desirable class of metrics of interests, especially in dimension $4$: for instance, they are the fixed points of the volume-preserving Hamilton--Ricci flow, and, as pointed out in~\cite[Section~11]{LeB99}, it is tantalizing to imagine that Einstein four-dimensional pieces play the same role as hyperbolic pieces do in the case of Thurston's geometrization conjecture. We refer to the survey paper of Lott and Kleiner~\cite{KL08} on the monumental work of Perelman resolving the Geometrization conjecture~\cite{Perelman1, Perelman2}.
	%\MG{Added Perelman.}
	
	We now describe Anderson's construction in some detail. One starts with any orientable $n$-manifold $N$ with cusps, equipped with a complete real-hyperbolic metric $g_{-1}$ of finite volume. Note that $g_{-1}$ is automatically Einstein with cosmological constant $-(n-1)$: $\Ric_{g_{-1}}=-(n-1)g_{-1}$. The manifold $N$ has finitely many cusps $\Ends(N)$, and we suppose that each of them is toral. Given any subset $\mathcal{E} \subset \Ends(N)$, we can perform Dehn filling along each end in $\mathcal{E}$ by filling a simple closed geodesics $\sigma_E$ in the ideal boundary of each $E \in \mathcal{E}$. For simplicity, let
	\[
	\bar{\sigma}=(\sigma_1, \dots, \sigma_p)
	\]
	be the collection of filled simple geodesics, and we denote by $N_{\bar{\sigma}}$ the corresponding filled manifolds. Note that $N_{\bar{\sigma}}$ is closed if and only if $\mathcal{E} = \Ends(N)$, in other words if we fill all the cusps of $N$. In this case, $N_{\bar{\sigma}}$ is a Dehn filling compactification. As above, we say that the Dehn filling is sufficiently large, if each geodesic $\sigma_i\in\bar{\sigma}$ is sufficiently long. Anderson's Dehn filling theorem then tells us that, if a Dehn filling is sufficiently large, then it admits a negative Einstein metric that is asymptotic to the real-hyperbolic metric on the ends which are not filled. The proof of this remarkable theorem is analytical in nature. First, one needs to construct a clever Riemannian metric on the Dehn-filled manifold $N_{\bar{\sigma}}$ which is \emph{approximately} Einstein. Then, a lengthy and technical argument using the implicit function theorem produces an \emph{exact} Einstein metric on $N_{\bar{\sigma}}$ which is close to the original approximate Einstein metric in some H\"older topology. We refer to~\cite{And06, Biquard1, Bamler} for further details. 

	Anderson's theory produces large classes of both compact and complete Einstein manifolds in dimensions $n\geq 4$. It is tantalizing to ask to what extent this theory could be extended by varying the geometry on the manifold that we want to fill, but still obtaining an Anderson-type result. One of the categories that was considered for such a construction is that of complex-hyperbolic manifolds.

	Recall that a complex-hyperbolic manifold is a complex manifold, and in particular it is an $n$-dimensional manifold with $n$ even. When $n=4$, we speak about \emph{complex-hyperbolic surfaces}, rather than about complex-hyperbolic $4$-manifolds. A finite-volume complex-hyperbolic surface is a quotient of 
\begin{align}\label{def1}
\IB^2:=\{|z_1|^2+|z_2|^2<1, \quad (z_1, z_2)\in\IC^2\}
\end{align}
equipped with the symmetric Bergman K\"ahler metric
\begin{align}\label{def2}
\omega_{\rm B}:=2i\bar{\partial}\partial \log(1-\|z\|^2)
\end{align}
by a torsion free non-uniform lattice $\Gamma$ in $\text{PU}(2, 1)$, which is the group of holomorphic isometries of $(\IB^2, \omega_{\rm B})$. It follows that any complete finite-volume quotient $\Gamma\backslash \IB^2$ is K\"ahler--Einstein with
\[
\Ric_{\omega_{\rm B}}=-\frac{3}{2}\omega_{\rm B},
\]
where by slightly abusing notation we denote by $\omega_B$ the induced K\"ahler metric on the quotient $\Gamma\backslash \IB^2$. If now $\Gamma$ is non-uniform, we then have that $\Gamma\backslash\IB^2$ is a non-compact, finite-volume, complete complex-hyperbolic surface with finitely many cuspidal ends. Each cuspidal end has an ideal boundary $N$, where $N$ is a compact infra-nilmanifolds, i.e., a compact orientable quotient of the $3$-dimensional Heisenberg Lie group. If all parabolic isometries in $\Gamma$ have
%\MG{``no rotation part''? ``no rotational part''? il trattino non mi ispira fiducia}
\emph{no rotational} part, then all of the $N$'s are nilmanifolds. (We point out that this condition on the parabolic isometries, and therefore on the cross-sections of the cusps, can always be attained by passing to a finite index normal subgroup in $\Gamma$. We point the reader to~\cite[Section~1.1]{JLMS} for more details.)

The folklore idea concerning complex-hyperbolic Einstein Dehn filling proceeds now in parallel with the real-hyperbolic case. Start with a complex-hyperbolic surface whose cusps have nilmanifold cross sections (recall in the real-hyperbolic case we required tori cross sections), truncate a sub-collection of these cusps, and glue back in the associated non-trivial disk bundles over tori. Note that a $3$-nilmanifold $N^3$ is diffeomorphic to a non-trivial $S^1$-bundle over a torus, which in turn is determined by the Euler number of the associated $D^2$-bundle. The complex orientation on a complex-hyperbolic surface induces an orientation on the boundaries of the corresponding truncated $4$-manifold. With this orientation, the $S^1$-bundles ideal boundaries of the cusps have \emph{positive} Euler number---note that these manifolds have no orientation-reversing self-diffeomorphisms. Viewed from the perspective of the Dehn-filled manifold $\overline X$, each of the 2-tori $T$ in the complement $\overline X \setminus X$ has \emph{negative} self-intersection.

Contrarily to the real-hyperbolic case, there is no choice involved for the filling.

%\MG{Nuovo enunciato}
\begin{proposition}\label{p:uniquefilling}
There is a unique way to fill a torus-like cusp of a complex-hyperbolic surface (as defined in \eqref{def1} and \eqref{def2}), up to diffeomorphism.
\end{proposition}

\begin{proof}
Recall that Dehn filling a torus-like cusp requires two steps: first we truncate the cusp, creating a boundary component $Y$, and then we glue onto the resulting boundary component a disc bundle $D$ over $T^2$ whose boundary is $-Y$. At each step we are making a choice, and we need to show that the diffeomorphism type of the Dehn filling is independent of these choices.

When we truncate the cusp, we are choosing the level at which we are truncating the end $Y \times [0,\infty)$. This choice is restricted by the complex-hyperbolic metric, so that every two choices differ by a translation in the $[0,\infty)$-direction. In particular, any two choices are smoothly isotopic, and it follows that the truncated manifold is unique up to diffeomorphism.

We now need to show that the gluing diffeomorphism does not affect the result of the filling, either. In order to see that, we view the disc bundle $D$ as the union of a (fibered) neighborhood $F$ of a fiber and its complement $H$. Therefore, gluing $D$ amounts to first gluing $F$ and then gluing $H$. We view $F$ as a 4-dimensional 2-handle: its gluing data is specified by the attaching curve (a simple closed curve in $Y$) and a framing. However, every self-diffeomorphism of a 3-dimensional nilmanifold preserves the $S^1$-fibration up to isotopy~\cite{WaldhausenI, WaldhausenII}, so we can choose the attaching curve to be a fiber. The framing, too, is determined by the fibration. Now, $H$ is a disk bundle over a torus minus a disk, so it is a 4-dimensional 1-handlebody. Laudenbach and Poenaru have shown that each self-diffeomorphism of the boundary of $H$ extends to $H$~\cite{LaudenbachPoenaru}, so the gluing diffeomorphism of $Y$ extends to $D$, as claimed.
\end{proof}

%In this case, we also say that $\overline X$ is a \emph{toroidal compactification} of $X$, and we call $\overline X \setminus X$ the \emph{compactifying divisor}

Given this general set-up, the hope was then to be able to at least replicate some of Anderson's arguments in order to produce new classes of Einstein manifolds in dimension four. Theorems~\ref{main-compact} and~\ref{main-noncompact} above show that such a plan is in general doomed for topological reasons. With that said, such a construction is not obstructed up to a finite cover, and it is tantalizing to ask whether a Dehn filling compactification can be performed up to a finite cover, and if the hypothetical Einstein metric coincides with the K\"ahler--Einstein metric constructed by Yau. Recall that up to a finite cover any finite-volume complex-hyperbolic surface with cusps admits a smooth toroidal compactification with ample canonical class, see \cite[Theorem A]{DiCerbo}.

	\section{Einstein metrics on $4$-manifolds, Hitchin--Thorpe-type inequalities, and $L^2$-characteristic numbers}\label{mainthm}
	
	In this section, we recall some results concerning Einstein metrics in dimension $4$ and $L^2$-characteristic numbers.
	We follow the notation and curvature normalization adopted in LeBrun's survey paper~	\cite{LeB99}. We also refer to~\cite{LeB99} as a general reference for the geometry and topology of $4$-manifolds and Einstein metrics.
	
	The Gauss--Bonnet formula for the Euler characteristic of a closed (oriented) $4$-manifold $(M, g)$ is given by
	\begin{align}\label{euler}
		\chi(M) = \frac{1}{8\pi^2} \int_{M} \Big(|W^+|_g^2 +|W^-|_g^2+\frac{s_g^2}{24}-\frac{|\overset{\circ}{\Ric}|_g^2}{2}\Big) d\mu_g,
	\end{align}
	where $W^{\pm}$ are the self-dual and anti-self-dual Weyl curvatures, $s_g$ is the scalar curvature, and $\overset{\circ}{\Ric}$ is the trace-free part of the Ricci tensor.
	Observe that $g$ is Einstein if and only if $\overset{\circ}{\Ric}=0$. Also, by the Hirzebruch signature theorem we can express the signature of $(M, g)$ as a curvature integral:
	\begin{align}\label{signature}
	\tau(M)=\frac{1}{12\pi^2}\int_{M}\big(|W^+|_g^2-|W^-|_g^2\big)d\mu_g.
	\end{align}
	
	Next, we recall the Hitchin--Thorpe inequality.
	
	\begin{theorem}[{\cite{Hit74}}]\label{hitchin}
		Let $(X, g)$ be a compact orientable Einstein $4$-manifold with signature $\tau$ and Euler characteristic $\chi$. Then
		\[
	     \chi\geq\frac{3}{2}|\tau|.
		\]
		Furthermore, if equality occurs then either $X$ is flat or the universal cover of $X$ is a $K3$ surface (up to orientation reversal).
	\end{theorem}
	
	The proof of the inequality in Theorem~\ref{hitchin} follows easily by combining Equations~\eqref{euler} and~\eqref{signature}, and this was originally observed in \cite{Thorpe}. The characterization of the equality case in such an inequality is more delicate, and we refer to~\cite{Hit74} for a beautiful proof of this fact.
	
	The formulas given in Equations~\eqref{euler} and~\eqref{signature} admit some interesting generalizations to the complete finite-volume setting. These equations will be used in Section~\ref{alternative}, where we prove Theorem~\ref{main-noncompact}. First, Equation~\eqref{euler} generalizes to the finite-volume setting with bounded curvature: this follows from the Gauss--Bonnet-type formula due to Cheeger and Gromov~\cite{CG85} (see also Harder \cite{Har71}):
	\begin{align}\label{eulerl2}
	\chi_{L^2}(M):=\frac{1}{8\pi^2}\int_{M}\Big(|W^+|_g^2
	+|W^-|_g^2+\frac{s_g^2}{24}-\frac{|\overset{\circ}{\Ric}|_g^2}{2}\Big)d\mu_g=\chi(M).
	\end{align}
	In this setting, one may also consider the $L^2$-curvature integral analogous to Equation~\eqref{signature}:
	\begin{align}\label{signaturel2}
	\tau_{L^2}(M)=\frac{1}{12\pi^2}\int_{M}\big(|W^+|_g^2-|W^-|_g^2\big)d\mu_g.
	\end{align}
	Despite the apparent similarity with Equation~\eqref{signature}, this finite curvature integral has no immediate topological interpretation in general. However, if the ideal boundary of $(M, g)$ is a finite collection of fibered cusps, with each fibration being a circle bundle over a surface, we do have a topological interpretation of Equation~\eqref{signaturel2} by keeping into account the limit $\eta$-invariant of each cusp. More precisely, we have the following formula, due to Dai and Wei (see~\cite[Page~568]{DW07}):
	%\MG{Non c'\`e bisogno di mettere ``Novikov'': le 4-variet\`a di tipo finito hanno sempre una ``buona'' segnatura, senza bisogno di ipotesi di compattezza o chiusura. Ho aggiunto qualche frase sotto. Tra l'altro, sei sicuro dei segni?}
	\begin{align}\label{signature2}
	\tau_{L^2}(M)=\tau(M)+\sum_{E \in \Ends(M)}\frac{e(E)-3\sign(e(E))}{3},
	\end{align}
	where $\tau(M)$ is again the signature of $M$, and where $e(E)$ is the Euler number of the circle bundle associated to $E\in \Ends(X)$. Recall that in order for the signature to be defined one just needs that $M$ is oriented and that $H_2(M;\IZ)$ is finitely generated. (Note that $H_2(M;\IZ)$ is automatically finitely generated for a complex-hyperbolic manifold of finite volume.) Under these assumptions, $H_2(M;\IZ)/{\rm Tor}$ is a free Abelian group of finite rank, and the intersection product induces on it a (possibly degenerate) bilinear form: given two classes $A, B \in H_2(M;\IZ)/{\rm Tor}$, we can represent them by two transversely embedded compact surfaces, $F_A$ and $F_B$; the intersection $Q_M(A,B)$ is the count of their intersections with the sign determined by the orientation. Then we define $\tau(M)$ to be the signature of $Q_M$, i.e., the number of positive eigenvalues minus the number of negative eigenvalues of $Q_M$. (When $M$ is non-compact or has non-empty boundary, $Q_M$ could have non-degenerate kernel.)
	
	By combining these formulas, we have the following generalization of the Hitchin--Thorpe inequality.

	%\MG{Changed the notation for the statement, so that we don't have to number the ends. Old statement commented.}
	\begin{theorem}[{\cite{DW07}}]\label{dai}
		Let $(X, g)$ be a non-compact, complete Einstein $4$-manifold whose ideal boundary is a finite collection of fibered cusps, with each fibration being a circle bundle over a surface. We then have
		\[
		\chi(X)\geq \frac{3}{2} \left| \tau(X) + \sum_{E \in \Ends(X)} \frac{e(E) - 3\sign(e(E))}3\right|
		\]
		where, as above, $e(E)$ is the Euler number of the circle bundle associated to $E$. 
		Moreover, equality holds if and only if $(X, g)$ is a complete Calabi--Yau manifold.
	\end{theorem}
	 
%	\begin{theorem}[{\cite{DW07}}]\label{dai}
%		Let $(X, g)$ be a non-compact, complete Einstein $4$-manifold whose ideal boundary is a finite collection of fibered cusps, with each fibration being a circle bundle over a surface. We then have
%		\[
%		\chi(X)\geq \frac{3}{2} \left| \tau(X) - \sum_{i} \frac{1}{3}e_i + \sign(e_i)\right|
%		\]
%		where $e_i$ is the Euler number of the circle bundle associated to the $i^{\rm th}$ end. 
%		Moreover, equality holds if and only if $(M, g)$ is a complete Calabi--Yau manifold.
%	\end{theorem}

	 \section{Non-Einstein Dehn filling compactifications}\label{part1}

	We start by describing Hirzebruch's example~\cite{Hir84} of a complex-hyperbolic surface of finite volume with a smooth toroidal compactification of Kodaira dimension zero, whose cusps all have Euler number $1$. We will then tinker around with Hirzebruch's example to produce, for each $n > 0$, a complex-hyperbolic surface that has the same property, except that the cusps have all Euler number $e$.
	
	\subsection{Hirzebruch's example(s)}
	Let $\zeta = e^{\frac{\pi i}{3}}$ and $E_{\zeta} = \IC/\Lambda_\zeta$ be the elliptic curve associated to the lattice $\Lambda_\zeta = \IZ[1, \zeta]$. Consider the Abelian surface $A = E_\zeta \times E_\zeta$, with coordinates $(z,w)$. Since $E_\zeta$ is an elliptic curve with complex multiplication, the following four elliptic curves in $A$ are defined:
	\begin{align}\label{configuration}
	C_0 = \{w=0\}, \quad C_\infty = \{z=0\}, \quad C_1 = \{w=z\},\quad C_{\zeta} = \{w=\zeta z\}.
	\end{align}
	The four elliptic curves in Equation~\eqref{configuration} intersect transversely only in the point $(0, 0)\in A$. Moreover, by adjunction we know that each of them has self-intersection $0$. Blow up $A$ at $(0,0)$ to obtain $X$, and let $D$ be the proper transform of $C_0 \cup C_\infty \cup C_1 \cup C_\zeta$. we obtain a pair $(X, D)$ where $X$ is a blown-up Abelian surface and $D$ is a divisor consisting of four smooth disjoint elliptic curves, that we call $D_1$, $D_2$, $D_3$, $D_4$.
	
	The surface $X$ is a smooth toroidal compactification of $X\setminus D$, which in turn can be shown to be complex-hyperbolic by Tian and Yau's uniformization theorem~\cite{TY87}. It suffices to compare $\sum_i D_i^2$ with $3c_2(X) - c_1^2(X)$, if one has equality in the (logarithmic) Bogomolov--Miyaoka--Yau (log-BMY, for short) inequality
	\begin{equation}\label{log-BMY}
	-\sum_i D_i^2 \le 3c_2(X) - c_1^2(X),
	\end{equation}
	then $X \setminus \cup_iD_i$ is complex-hyperbolic.
	In our case, the left-hand side is $-4\cdot(-1)$, while the right-hand side is $3\cdot 1 - (-1)$, so we have equality, and $X\setminus \cup_iD_i$ is complex-hyperbolic, as claimed.
	By construction, $X \setminus D$ has four torus-like cusps, all with Euler number 1.
	
	In~\cite{Hir84}, Hirzebruch also consider $n^2$-fold covers of the surface $X$ we constructed above, in which he finds $4n^2$ elliptic curves, each of self-intersection $-n^2$, and one easily check that their complement again satisfies equality in Equation~\eqref{log-BMY}, and thus the complement of all these curves is a complex-hyperbolic surface with all cusps of Euler number $n^2$.
	
	\begin{remark}
	Hirzebruch then uses these surfaces and these divisors to produce compact complex surfaces approaching the equality in the (non-logarithmic) BMY inequality. We will not need this part of his paper, but just the most basic example.
	\end{remark}
	 
	\subsection{Playing with Hirzebruch's example}
	We want to tweak Hirzebruch's example from the previous subsection to obtain, for each $e > 0$, a complex-hyperbolic surface whose cusps all have Euler number $e$. (As mentioned above, Hirzebruch himself had already found such examples, when $e$ is a square.) In fact, we will construct such examples with four cusps.
	 
	We will present a covering argument in slightly broader generality.
	%\MG{Forse vogliamo che $A$ sia una superficie abeliana---quello che ci serve ``e una propriet\`a stabile per rivestimenti \'etale, e non so se essere un prodotto lo sia. (e magari non ci serve comunque che lo sia.)}
	Consider a pair $(A,C)$ where $A$ is a product of two tori $E \times E'$ and $C$ is the union of four tori $C_1, C_2, C_3, C_4$, such that $C_j \cdot C_k \neq 0$ whenever $j\neq k$.
	 
	\begin{lemma}\label{l:primecover}
		For each prime $p$ there exists a $p$-fold \'etale cover $\pi \colon A' \to A$ such that $C_j' := \pi^{-1}(C_j)$ is connected for $j = 1,\dots, 4$. For such a cover $\pi$, $C'_j \cdot C'_k = p \cdot (C_j \cdot C_k)$.
	\end{lemma}
	
	\begin{proof}
	Let $\IF_p$ the field with $p$ elements. Since each $C_j$ is a torus and every two of them intersect non-trivially, $P_j := H_1(C_j;\IF_p) \subset H_1(A;\IF_p)$ is a $2$-plane in $\IF_p$, and $P_j \cap P_k = \{0\}$ whenever $j\neq k$.
	%\MG{Corretto $P_j \cap P_k$ invece di $P_j \cap P_j$.}
	 	
	In order to prove the lemma, we need to find a surjective homomorphism $\phi \in \Hom(H_1(A;\IZ), \IF_p) \cong H^1(A;\IF_p) \cong \IF_p^4$ such that $\ker \phi \cap P_j$ is 1-dimensional for each $j$. To see that this is sufficient, fix one of the divisors, $C_j$. Since $\ker \phi$ does not contain all of  $P_j$, there is a curve $\gamma \subset C_1$ such that $\phi([\gamma]) \neq 0 \in \IF_p$. Since $p$ is a prime, $\phi([\gamma])$ generates $\IF_p$, so $\gamma$ lifts
	%\MG{it lifts $\to$ $\gamma$ lifts}
	to a single simple closed curve in $\pi^{-1}(X)$, and in particular $\pi^{-1}(C_j)$ is connected.
	 	
	%We now look for the subspace $\ker \phi$, which will give $\phi$ by passing to the quotient. We need to find a hyperplane in $H_1(A;\IF_p) \cong \IF_p^4$ that does not contain any of $P_1, \dots, P_4$. Let us count hyperplanes in $H_1(X;\IF_p)$ that \emph{do} contain at least one of them. In fact, since $P_j \cap P_k = \{0\}$, no hyperplane can contain both $P_j$ and $P_k$. Since $P_j$ is a codimension-2 subspace, there are exactly $p+1$ hyperplanes containing it (here $p+1$ is the cardinality of $\IP^1_{\IF_p}$, which parametrizes all such hyperplanes). Thus, there are $4(p+1)$ hyperplanes containing one among $P_1,\dots,P_j$. However, the set of hyperplanes in $H_1(A;\IF_p)$ comprises $p^3+p^2+p+1$ elements (it is a finite projective space, namely $\IP^3_{\IF_p}$).
	
	%\MG{Molto pi\`u esplicito quello che sto dicendo, ma comunque $\phi$ non lo costruisco esplicitamente.}
	Up to scalars, a homomorphism $\phi \colon H_1(A;\IZ) \to \IF_p$ is determined by its kernel, $\ker \phi$. Conversely, if we find a hyperplane $K \subset  H_1(A; \IF_p)$ such that $K \cap P_j$ is 1-dimensional for each $j$, then there exists $\phi_K \in H^1(A;\IF_p)$ such that $\ker \phi = K$, which is the homomorphism we required.
	
	Instead of exhibiting such a hyperplane $K$, we will non-constructively show that there exists one by counting: we will prove that the number of hyperplanes in $H_1(A; \IF_p)$ such that $K \cap P_j$ is \emph{not} 1-dimensional for at least an index $j \in \{1,2,3,4\}$ is strictly smaller than the number of hyperplanes in $H_1(A; \IF_p)$. (The total number of hyperplanes in $H_1(A; \IF_p)$ is the cardinality of the projective space $\IP(H^1(A; \IF_p)) \cong \IP^3_{\IF_p}$, which is $p^3+p^2+p+1$.)
	
	Let us now consider a hyperplane $K \subset H_1(A; \IF_p)$. Since $K$ is a hyperplane and $P_j$ is a 2-plane, $1\le \dim(K \cap P_j) \le 2$. Moreover, since $P_j \cap P_k = \{0\}$ whenever $j \neq k$, for a given hyperplane $K$ there is at most an index $j$ such that $\dim(K \cap P_j) = 2$, or, equivalently, such that $P_j \subset K$.
	
	The number of hyperplanes $K$ containing $P_j$ for a fixed index $j$ is $p+1$: containing the 2-plane $P_j$ means imposing two linear conditions in the projective space $\IP(H^1(A; \IF_p))$, so the set of hyperplanes containing $P_j$ is a projective line in $\IP(H^1(A; \IF_p))$, which is a subset of cardinality $\# \IP^1_{\IF_p} = p+1$.
	
	Therefore, there are $4(p+1)$ hyperplanes containing one among $P_1,\dots,P_j$. Since
		\[
		p^3+p^2+p+1 - 4(p+1) = p^3+p^2-3p-3 = (p^2-3)(p+1) > 0
		\]
	for every $p > 1$, there is at least a hyperplane $K$ in $H_1(A; \IF_p)$ that does not contain any among $P_1, \dots, P_4$, so in particular $\dim(K \cap P_j) = 1$ for $j = 1,\dots, 4$, as required.
	 	
	The second part of the statement is obvious by degree considerations.
	\end{proof}
	 
	\begin{corollary}\label{c:cover}
		For every integer $e > 1$ there exists an $e$-fold \'etale cover of $A$ such that the lifts of $C_1, \dots, C_4$ are all connected.
	\end{corollary}
	 
	\begin{proof}
		We argue by induction on the number of (non-distinct) prime factors of $e$; the base case is when $e$ is prime, and this was proved in the previous lemma. If $e$ is not prime, choose a prime $p$ dividing it. By the previous lemma and by the inductive assumption, there is an $\frac{e}p$-fold cover $A'$ of $A$ in which $C$ lifts to four tori $C'$. Note that $C'$ still satisfies the assumptions of Lemma~\ref{l:primecover}. By Lemma~\ref{l:primecover}, there is a $p$-fold cover $A''$ of $A'$ in which $C'$ lifts to four tori $C''$. Now $(A'', C'')$ is the required cover.
	\end{proof}
	 
	\begin{remark}
		With a little more effort one can find a \emph{cyclic} cover in Corollary~\ref{c:cover}; it suffices to take a $e^4$-fold cover $A'' \to A$, as given by the corollary, and then observe that this corresponds to a subgroup $H \subset H_1(X;\IZ)$ whose quotient $G = H_1(X;\IZ)/H$ has order $e^4$. Since $H_1(A;\IZ) \cong \IZ^4$, $G$ is generated by (at most) four elements, one of which has to have order divisible by $e$. In particular, $G$ has itself a cyclic
		%\MG{Aggiunto `\emph{cyclic} quotient $G'$', che mi sa che era quello che volevo dire}
		quotient $G'$ of order exactly $e$. Composing the maps $H_1(X;\IZ) \to G \to G'$ gives the desired cover.
	\end{remark}
	 
	%\MG{By the way, I'm not particularly attached to the $^{(e)}$ notation. Maybe we should just call $C^0, C^1, C^\zeta, C^\infty$ the first four, and then use normal indices for the rest.}
	%\MG{Ok, I decided to remove the brackets from $A^{(e)}$ and I also corrected a couple of $n$s that were still there (instead of $e$s).}
	Let us now start with Hirzebruch's four tori $C_0, C_\infty, C_1, C_\zeta \subset A = E_\zeta \times E_\zeta$. Call $(A^{1}, C^{1}) := (A, C_0 \cup C_\infty \cup C_1 \cup C_\zeta)$. Fix a positive integer $e$. By the corollary above, there exists an $e$-fold \'etale cover $\pi \colon A^{e} \to A^{1}$ such that $C^{e} := \pi^{-1}(C^{1})$ is the union of four elliptic curves, each with self-intersection $0$, and such that each two of them intersect pairwise $e$ times. Moreover, since $\pi$ is an \'etale cover, all $e$ intersection points are quadruple points where four distinct tori meet transversally.
	%\MG{Aggiunte due frasi sui numeri caratteristici di $A^e$.}
	The $4$-manifold $A^{e}$ is an \'etale cover of the 4-torus $A^1$, so it is itself a 4-torus, therefore $c_2(A^e) = 0$ and $\tau(A^e) = 0$. The complex surface $A^e$ has first Chern class which is the pull-back of $c_1(A^1)$, so $c_1(A^e) = \pi^*(c_1(A^1)) = \pi^*(0) = 0$.
	 
	%\MG{Aggiunti dettagli dei conti, tipo cosa succede quando scoppi.}
	Blow up at these $e$ intersection points, to obtain $\tilde A^e$. Since $\chi(A^{e}) = 0$, $c_1^2(A^{e}) = 0$, and $\tau(A^{e}) = 0$, and since each blow-up increases $c_2$ by $1$, decreases $c_1^2$ by $1$ (see, for instance,~\cite[Theorem~I.9.1]{BHPV}) and $\tau$ by $1$, we have that $c_2(\tilde A^e) = \chi(\tilde A^e) = e$, $c_1^2(\tilde A^e) = -e$, and $\tau(\tilde A^e) = -e$. The proper transform $D^e$ of $C^{e}$ consists of four disjoint tori $D^e_{1},\dots, D^e_{4}$, each of self-intersection $-e$, and in particular:
	\[
	-\sum_j D^e_{j}\cdot D^e_{j} = 4e = 3e - (-e) = 3c_2(\tilde A^e) - c_1^2(\tilde A^e),
	\]
	so $X_e := \tilde A^e \setminus D^e$ is a complex-hyperbolic surface whose cusps all have Euler number $e$.
	
	\subsection{The proof of Theorem~\ref{main-compact}}
	
	The following proposition directly implies Theorem~\ref{main-compact}.
	
	\begin{proposition}
	%\MG{new statement and modified proof.}
	Fix a positive integer $e$. Any $4$-manifold $\overline X$ homeomorphic to the Dehn filling compactification $\tilde A^e$ of $X_e$ violates the Hitchin--Thorpe inequality, and in particular it carries no Einstein metric.
	\end{proposition}
	
	\begin{proof}
	We compute the Euler characteristic and signature of $\tilde A^e$: $|\tau(\tilde A^e)| = \chi(\tilde A^e)$ by the computations above, so it violates the Hitchin--Thorpe inequality. Any 4-manifold $\overline X$ homeomorphic to $\tilde A^e$ has the same Euler characteristic and signature, so it also violates the Hitchin--Thorpe inequality. It follows that neither carries an Einstein metric.
	\end{proof}

	\section{Non-Einstein non-compact Dehn fillings}\label{alternative}
	%\MG{`non-compact' instead of `partial'?}

	Let us once again use Hirzebruch's example as a starting point. We start with a sequence of Abelian varieties $\{A_e\}, e\in \IN$ defined as follows
	\[
	A_e:= \IC/\IZ[1, \zeta]\times \IC/\IZ[e, \zeta].
	\]
	Notice that for $e=1$ we recover the Abelian surface considered by Hirzebruch. Inside the Abelian surface $A_e$ with coordinates $(z, w)$ consider the elliptic curves
	\[
	w=0, \quad w=z, \quad w=\zeta z.
	\]
	These curves meet transversally at $e$ distinct points
	\begin{align}\label{intersecpoints}
	(0, 0), \quad (0, 1), \quad \dots, \quad (0, e-1).
	\end{align}
	Consider also $e$ vertical elliptic curves in $A_e$ of equations
	\[
	z=0, \quad z=1, \quad \dots, \quad z=e-1.
	\]
	(In the language of the previous section, this is an $e$-fold cyclic cover of Hirzebruch's example, but this time one of the curves has disconnected pre-image.)
	In other words, we have a configuration of $e+3$ elliptic curves meeting transversally at the points by Equations \eqref{intersecpoints}. We can now blow up these points to get a surface $Z_{e}$ birational to $A_e$ with $\chi(X_e)=e$. The proper transforms of the elliptic curves in $A_e$ described above, we obtain $e+3$ disjoint elliptic curves 
	\[
	D_1, \quad D_2,\quad D_3, \quad D_4, \quad \dots , \quad D_{e+3},
	\]
	where
	\[
	D_1^2=D_2^2=D_3^2=-e, \quad D_4^2=D_5^2=\dots=D^2_{e+3}=-1.
	\]
	Let us denote by $D^e$ the reduced divisor corresponding to the union of all of the $D_j$'s. We then compute that the self-intersection of the log-canonical divisor of the pair $(Z_e, D^e)$ satisfy
	\begin{align*}
	(K_{Z_e}+D^e)^2&=K^2_{Z_e}-(D^e)^2=-e+e+e+e+1+...+1=3e \\ \notag
	&=3\chi(Z_e)=3\chi(Z_e\setminus D^e).
	\end{align*}
	Thus, the pair $(Z_e, D^e)$ saturates the log-BMY inequality and as a result $Y_e := Z_e\setminus D^e$ is biholomorphic to a complex-hyperbolic surface with $e+3$ cusps, three of which have Euler number $e$.
	
	We are now in position to proving Theorem~\ref{main-noncompact}. We will prove the following, slightly stronger statement.
	
	%\MG{Piccolo problema: quando $e = 1$ (cio\`e l'esempio di Hirzebruch), riempiendo tre cuspidi ne resta soltanto una fuori, quindi l'argomento non funziona. Per\`o l'argomento funziona riempiendone una sola, credo :), che forse \`e ancora meglio... Per ora ho messo $e>1$ nell'enunciato.}
	
	\begin{proposition}
	%\MG{Minimal edits on the statement and in the proof.}
	Let $e$ be a positive integer. If $e > 1$, the $4$-manifold obtained by Dehn filling the three cusps of $Y_e$ with Euler number $e$ admits no complete Einstein metric with fibered cusp structure at infinity. If $e=1$, the $4$-manifold obtained by Dehn filling the three cusps of $Y_1$ admits no complete Einstein metric with fibered cusp structure at infinity and with negative Einstein constant (e.g., asymptotic to the complex-hyperbolic metric). %The same is true up to homeomorphism.
	\end{proposition}
	
	\begin{proof}
	%\MG{New argument to compute $\tau$---the old one is commented.}
%	By using the Dai and Wei's formula, Equation~\eqref{Dai}, we obtain
%	\[
%	\tau(X_e)+\frac{1}{3}e-sign(e)=\tau_{L^2}(X_n\setminus D^n)=\frac{n}{3},
%	\]
%	so that 
%	\[
%	\tau(M_n)=\frac{n}{3}+\frac{2}{3}n-\frac{n}{3}-\frac{n}{3}-\frac{n}{3}+1+1+1.
%	\]
%	Bottom line for any $n$ we have $\tau(M_n)=3$.
%	\MG{stiamo comunque usando l'additivit\`a di novikov. metto la citazione a kirby.}
	Since $Y_e$ is the complement of $e+3$ pairwise disjoint divisors of negative self-intersections, by Novikov additivity (see for instance~\cite[Section~II.5]{Kirby-4mflds}) we have
	\[
	\tau(Y_e) = \tau(Z_e) + (e+3) = -e + e+3 = 3.
	\]
If we now fill the three cusps of $Y_e$ with Euler number $e$, we get a manifold, call it $\overline Y$, with $e$ cusps. Again by Novikov additivity, $\tau(\overline Y)=0$.
	%\MG{Should the manifold be $\overline{Y}$ instead of $X_e$? (E credo che possiamo anche eliminare il divisore $D^e$, tanto chissenefrega)}
	If this is the case we can now obtain a contradiction if we assume we were able to Einstein Dehn fill these three cusps.
	
	%\MG{Changed this paragraph---old version commented.}
	Let us call $\overline Y$ the manifold obtained from $(X_e, D^e)$ by filling the cusps $D_1, D_2, D_3$. Then $\overline Y$ has $e$ cusps with Euler number 1, $\chi(\overline Y)=e$, and $\tau(\overline {Y})=0$, so $\overline Y$ attains equality in the Dai--Wei inequality of Theorem~\ref{dai}. In particular, by Theorem~\ref{dai}, if it supports an Einstein metric with fibered structure at infinity, this metric is Ricci flat.

%	If so, the manifold obtained from $(X_e, D^e)$ by filling the cusps $D_1, D_2, D_3$ has an Einstein metric and again by the Dai--Wei inequality of Theorem~\ref{dai} we have:
%	\begin{align}\label{noncompactHitchin}
%	2\chi(\overline Y)+3\tau(\overline Y) + \sum^e_{j=1}e(E_j) - 3\sign(e(E_j))\geq 0,
%	\end{align}
%	with equality if and only if the Einstein metric is Calabi--Yau (Ricci flat). In our case, $\chi(\overline Y)=e$, $\tau(\overline {Y})=0$, and all unfilled cusps have Euler number $1$. 
%	Thus, we get equality in Equation~\eqref{noncompactHitchin}.
	%\MG{Io scambierei la parte $e=1$ con la parte $e > 1$: sostanzialmente per la parte $e=1$ abbiamo gi\`a finito qui, mentre c'\`e solo da lavorare se $e > 1$. Che ne dici?}
	Now, since $e>1$, after filling the three cusps with Euler number $e$, we have that $\overline{Y}$ has at least two ends. Given any of the two non-compact ends, we can construct a line, i.e., a geodesic $\gamma: (-\infty, \infty)\to \overline{Y}$ such that $d(\gamma(t), \gamma(s))=|t-s|$ for any $s, t\in\IR$. By the Cheeger--Gromoll splitting theorem~\cite{CG71} such a line splits isometrically, and as a result our partially compactified manifold should have topological Euler characteristic equal to zero. On the other hand, $\chi(\bar{Y})=e$ and we then get a contradiction. Finally, if $e=1$, an Einstein metric asymptotic to a complex-hyperbolic one cannot be Ricci flat: the cosmological constant of the complex hyperbolic metric is negative and not zero as for Calabi--Yau surfaces.
	\end{proof}
	
	%\MG{To add: the new argument with the splitting lemma. Maybe a remark/proposition/proof about the two-cusped case that we discussed last time?}
	%%%%%%%%%%%%%%%%%%%%

\section{Final remarks and questions}

%\MG{New question.}
We conclude with some remarks and questions concerning Theorems~\ref{main-compact} and~\ref{main-noncompact}. In light of the theorem, one is naturally lead to ask the following question.

\begin{question}\label{q:onecusp}
%\MG{minor edits in the question}
	Fix an integer $e > 0$. Does there exist a complex-hyperbolic surface $Z_e$ with a torus-like cusp with Euler number $e$ such that Dehn filling of $Z_e$ does not support an Einstein metric with fibered cusp at infinity?
\end{question}

A positive answer would immediately imply that there is no single cusp-filling procedure that starts from one torus-like cusps of a complex-hyperbolic surface and produces an Einstein metric. In this sense, Theorem~\ref{main-compact} proves the impossibility of complex-hyperbolic Dehn filling in a more constrained setting, namely when we require the Dehn filling to produce a closed orientable manifold. On the other hand, we could ask that the (hypothetical) metric produced by such an operation is itself asymptotic to a standard cusp-like metric at infinity. This is indeed true of the metrics produced by Thurston and by Anderson, where in fact the metric at the cusp is asymptotic to the original metric, thus allowing for the construction to be iterated~\cite{And06}. If we impose the further constraint that such an Einstein metric is asymptotic to the initial complex-hyperbolic metric, the proof of the second statement simplifies, and we can also produce answers to Question~\ref{q:onecusp} whenever $e$ is divisible by $3$. (These examples are 2-cusped, so after Dehn filling one cusp we cannot use the splitting theorem as we do for $X_e$.)

%\MG{new paragraph and question. forse ho scritto un po' troppe cazzate, per\`o :)}
As mentioned at the end of Section~\ref{s:background}, thanks to~\cite[Theorem A]{DiCerbo} and Yau's celebrated result~\cite{YauFields}, we know that complex-hyperbolic K\"ahler--Einstein Dehn filling can be performed if we pass to a finite cover. The proof uses the ampleness of the canonical divisor of the toroidal compactification of a (sufficiently large) finite cover. It is interesting to ask whether such an argument holds for metric, rather than complex-algebraic reasons, as in Anderson's construction. More precisely, we ask the following question.

\begin{question}
Does there exists a constant $L$ such that, if the fiber of a torus-like cusp $c$ of a complex-hyperbolic surface $X$ has length larger than $L$, then the 4-manifold obtained by Dehn filling $X$ at $c$ supports an Einstein metric that is not K\"ahler--Einstein?
\end{question}

This would show that, given a complex-hyperbolic surface, one could find a finite cover in which the fiber of a cusp gets sufficiently long so as to assure that the corresponding Dehn filling is Einstein. In a way, the choice of the cover makes up for the lack of choice of slope of the filling. In this sense, this hypothetical construction would resemble Fine and Premoselli's branched covering construction \cite{FP20}.
%Finally, it is tantalizing to ask if similar obstructions can be derived in higher dimensions. This seems to be a quite tricky question, and an approach similar to the one we presented here appears to be currently out of reach. Indeed, no obstructions for the existence of an Einstein metric with negative cosmological constant on manifolds of real dimension $n\geq 5$ are currently known! Also, to make things more complicated, smooth toroidal compactifications of complex-hyperbolic manifolds of complex dimension $n\geq 6$ always support a K\"ahler--Einstein metric \cite{YauFields, BakkerTsimerman}, as explained in Section \ref{intro}. 
%\textcolor{red}{Thus, a potential proof of the fact that Einstein Dehn filling compactifications cannot exist in higher dimensions will have to rely upon obstructions depending on topological invariants that distinguish different Dehn filling compactifications.}

\bibliography{ComplexDehn}

\bibliographystyle{amsalpha}
	
\end{document}